\documentclass[a4paper, 11pt, reqno, final]{amsart}

\usepackage{amsmath,amssymb,amsthm}

\numberwithin{equation}{section}
\allowdisplaybreaks

\theoremstyle{definition}
 \newtheorem{thm}{Theorem}[section]
 
 \newtheorem{lem}[thm]{Lemma}
 \newtheorem{cor}[thm]{Corollary}
 
 \newtheorem{fct}[thm]{Fact}
 
 \newtheorem{eg} [thm]{Example}
 \newtheorem*{ack}{Acknowledgements}
 \newtheorem*{not*}{Notation}
 \newtheorem*{rmk*}{Remark}
 \newtheorem*{thm*}{Theorem}

\newcommand{\bbC}{\mathbb{C}}

\newcommand{\bbQ}{\mathbb{Q}}
\newcommand{\bbZ}{\mathbb{Z}}
\newcommand{\calF}{\mathcal{F}}

\newcommand{\calH}{\mathcal{H}}
\newcommand{\calL}{\mathcal{L}}

\newcommand{\calW}{\mathcal{W}}

\newcommand{\frksl}{\mathfrak{sl}}
\newcommand{\SU}{\mathrm{SU}}

\newcommand{\ep}{\epsilon}
\newcommand{\ve}{\varepsilon}
\newcommand{\rk}{\mathrm{rank}}
\newcommand{\Vir}{\mathrm{Vir}}

\newcommand{\seteq}{\mathbin{:=}}
\newcommand{\ket}[1]{\left|  #1  \right>}
\newcommand{\bra}[1]{\left< #1 \right|}
\newcommand{\nop}{\genfrac{}{}{0pt}{1}{\circ}{\circ}}
\newcommand{\no}[1]{\nop #1 \nop}
\newcommand{\vect}[1]{\overrightarrow{#1}}
\makeatletter
\def\Res{\mathop{\operator@font Res}}
\def\Ind{\mathop{\operator@font Ind}\nolimits}
\def\End{\mathop{\operator@font End}}
\makeatother

\title[Logarithmic primary of Virasoro algebra]{Norm of logarithmic primary of \\ Virasoro algebra}
\date{October 17, 2010; revised June 11, 2011}
\author{Shintarou Yanagida}

\address{Kobe University, Department of Mathematics, Rokko, Kobe 657-8501, Japan}
\email{yanagida@math.kobe-u.ac.jp}

\begin{document}

\begin{abstract}
We give an algebraic proof of the formula on the norm of logarithmic primary of Virasoro algebra, which was proposed by Al.~Zamolodchikov. 
This formula appears in the recursion formula for the norm of Gaiotto state, 
which guarantees the AGT relation for the four-dimensional 
SU(2) pure gauge theory. 
\end{abstract}

\maketitle

\section{Introduction}

This paper discusses algebraic or combinatorial 
calculation of certain elements in the Verma module of Virasoro algebra.
Our main result is Theorem \ref{thm:main} stated in \S \ref{subsect:main},
where a mathematcal proof of the formula on the norm of logarithmic primary of Virasoro algebra is given. 
This formula was proposed in \cite{Z:2004},
and the proof \cite{HJS:2010} of the AGT relation for pure SU(2) gauge theory \cite{G:2009} 
depends on it. 
The result of this paper is the last piece of the proof of the AGT relation.

However, in order to state that, we need to introduce several notations and recall 
well-known facts on Virasoro algebra.
Subsections \S \ref{subsec:vir} and \S \ref{subsec:efsv} are 
devoted to these preliminaries.
Such topics are often treated in the textbooks of conformal field theory, 
such as \cite{FMS:1997}, \cite{ID:1989} and \cite{S:2003}.
The reader who is familiar with Virasoro algebra 
may skip to \S \ref{subsec:nlp}.

In \S \ref{subsec:nlp} we introduce the norm of logarithmic primary, 
which is the main topic of this paper 
and appears directly in the statement of Theorem \ref{thm:main}.

In \S \ref{subsec:AGT} and \S \ref{subsec:prfAGT}
we mention to AGT conjectures/relations and its connection to our main theorem.
These subsections are a detour, 
but it will be interesting for those working on AGT conjectures.

In the last subsection \S \ref{subsect:main} of this introduction,
we state our main theorem and 
the contents of the main part of this paper.

\subsection{Virasoro algebra and singular vectors}\label{subsec:vir}

A singular vector in the Verma module of Virasoro algebra is a fundamental 
object in the two-dimensional conformal field theory and the representation 
theory of Virasoro algebra 
since its emergence in the classical paper \cite{BPZ:1984}.

Let us recall the definition of singular vectors, 
fixing notations on Virasoro algebra and its Verma module.
The Virasoro algebra $\Vir$ is the Lie algebra over $\bbC$ 
generated by $L_n$ ($n\in\bbZ$) and $C$ (central) with the relation
\begin{align}\label{eq:Vir}
[L_m,L_n]=(m-n)L_{m+n}+\dfrac{C}{12}m(m^2-1)\delta_{m+n,0},\quad
[L_n,C]=0.
\end{align}
$\Vir$ has the triangular decomposition 
$\Vir=\Vir_{+}\oplus\Vir_{0}\oplus\Vir_{-}$ with 
$\Vir_{\pm} \seteq \oplus_{\pm n>0}\bbC L_n$ and
$\Vir_{0} \seteq \bbC C \oplus \bbC L_0$.

Let $c$ and $h$ be complex numbers.
Let $\bbC_{c,h}$ 
be the one-dimensional representation 
of the subalgebra $\Vir_{\ge0} \seteq \Vir_{0}\oplus \Vir_{+}$,
where $\Vir_{+}$ acts trivially, $L_0$ acts as multiplication by $h$, 
and $C$ acts as multiplication by $c$.
Then the Verma module $M(c,h)$ is defined by
\begin{align*}
M(c,h)\seteq\Ind_{\Vir_{\ge0}}^{\Vir}\bbC_{c,h}.
\end{align*}
Obeying the notation in physics literature, 
we denote by $\ket{c,h}$ a fixed basis of $\bbC_{c,h}$.
Then one has $\bbC_{c,h}=\bbC \ket{c,h}$ and $M(c,h)=\Vir \ket{c,h}$.

$M(c,h)$ has an $L_0$-weight decomposition: 
$M(c,h)=\bigoplus_{n\in\bbZ_{\ge 0}} M(c,h)_{n}$ 
with
\begin{align*}
M(c,h)_{n}\seteq\{v\in M(c,h)\mid L_0 v=(h+n)v \}.
\end{align*}

A basis of $M(c,h)_{n}$ can be described by partitions.
In this paper, a partition of positive integer $n$ 
means a non-increasing sequence 
$(\lambda_1,\lambda_2,\ldots,\lambda_k)$
of positive integers such that $\sum_{i=1}^k \lambda_i=n$.
We also consider the empty sequence $\emptyset$ 
as a partition of the number $0$. 
The symbol $\lambda \vdash n$ means that $\lambda$ is a partition of $n$. 
We also define $|\lambda|\seteq\sum_{i=1}^k \lambda_i$.
For a partition 
$\lambda=(\lambda_1,\lambda_2,\ldots,\lambda_k)$ of $n$ 
we use the symbol 
\begin{align}\label{eq:L-lambda}
L_{-\lambda}\seteq L_{-\lambda_k} L_{-\lambda_{k-1}}\cdots L_{-\lambda_1},
\end{align}
an element of the enveloping algebra  $U(\Vir_{-})$ of the subalgebra $\Vir_{-}$.
Using these notations, one finds that the set
\begin{align*}
\{L_{-\lambda} \ket{c,h} \mid \lambda \vdash n\},
\end{align*}
is a basis of $M(c,h)_n$.

An element $v$ of $M(c,h)_n$ is called a singular vector of level $n$ if 
\begin{align}\label{eq:sing:def}
L_k v=0 \ \text{ for any } k\in\bbZ_{>0}.
\end{align}

The existence of singular vector restricts the values 
of the highest weights $(c,h)$.
To see this phenomena, it is necessary to recall the Kac determinant formula.

First, we introduce the (restricted) dual Verma module $M^*(c,h)$.
This is a right $\Vir$-module generated by $\bra{c,h}$ with 
 $\bra{c,h}\Vir_{+}=0$, $\bra{c,h} L_0=h\bra{c,h}$ 
and $\bra{c,h} C=c\bra{c,h}$.
It has an $L_0$-weight decomposition $\bigoplus_{n\in\bbZ_{\ge0}}M^*(c,h)_n$ 
with $M^*(c,h)_n\seteq\{v\in M^*(c,h) \mid v L_0=(h-n)v\}$.
$M^*(c,h)_n$ has a basis 
$\{\bra{h} L_{\lambda} \mid \lambda\vdash n\}$
with 
\begin{align}\label{eq:L+lambda}
L_{\lambda}\seteq  L_{\lambda_1} L_{\lambda_2} \cdots L_{\lambda_k}.
\end{align}
for the partition 
$\lambda=(\lambda_1,\lambda_2,\ldots,\lambda_k)$.

Next we introduce the contravariant form.
It is a bilinear map on the modules 
\begin{align}
\cdot: M^*(c,h) \times M(c,h) \to \bbC \label{eq:contrav}
\end{align} 
determined by 
\begin{align*}
&\bra{h}\cdot\ket{h}=1,\quad 
\bra{h}u_1u_2 \cdot\ket{h}=\bra{h}u_1 \cdot u_2\ket{h}
=\bra{h}\cdot u_1 u_2\ket{h}\ (u_1,u_2\in \Vir).
\end{align*}
We usually omit the symbol $\cdot$ and write 
$\bra{h}u\ket{h}\seteq \bra{h}u\cdot\ket{h}$
as in the physics literature.
By counitng $L_0$-weights   one can easily see that 
\begin{align}\label{eq:kac:block}
M^*(c,h)_m\cdot M(c,h)_n=0  \text{ if } m\neq n.
\end{align}
This bilinear form is contravariant in the following sense: 
\begin{align}
\bra{h}L_\lambda L_{-\mu}\ket{h}=\bra{h}L_\mu L_{-\lambda}\ket{h} 
\quad \text{ for any }\ \lambda,\ \mu. \label{eq:contrav:contrav}
\end{align}
It is usually called the contravariant form (or Shapovalov form) 
on the Verma module.

Recalling the basis \eqref{eq:L-lambda} and the dual basis \eqref{eq:L+lambda},
we define 
\begin{align*}
K_{\lambda,\mu}(c,h)\seteq \bra{c,h}L_{\mu}L_{-\lambda}\ket{c,h}.
\end{align*}
Then the properties of the contravariant form are encoded in the 
(infinite size) matrix $(K_{\lambda,\mu})$, where $\lambda$ and $\mu$ 
run over the set of all partitions.
However, because of \eqref{eq:kac:block},
we only need to consider the $p(n)\times p(n)$ matrix 
\begin{align}\label{eq:Kac:matrix}
K_n\seteq(K_{\lambda,\mu})_{\lambda,\mu\vdash n}
\end{align}
for each $n\in\bbZ_{\ge 0}$.
By \eqref{eq:contrav:contrav} it is a symmetric matrix.
Let us write down some examples:
\begin{align}
\label{eq:Kac:matrix:eg}
\begin{split}
K_1&=(K_{(1),(1)})=(2h),
\\
K_2&=\begin{pmatrix}K_{(1^2),(1^2)}&K_{(1^2),(2)} \\ 
                    K_{(2),(1^2)} & K_{(2),(2)}\end{pmatrix}
    =\begin{pmatrix}4h(1+2h)&6h \\ 6h & 4h+c/2\end{pmatrix},
\\
K_3&=\left(\begin{array}{lll}
                    K_{(1^3),(1^3)}&K_{(1^3),(2,1)}&K_{(1^3),(3)} \\ 
                    K_{(2,1),(1^3)}&K_{(2,1),(2,1)}&K_{(2,1),(3)} \\
                    K_{(3),  (1^3)}&K_{(3)  ,(2,1)}&K_{(3)  ,(3)} 
     \end{array}\right)\\
    &=\begin{pmatrix}24h(1+h)(1+2h)& 12h(1+3h)  & 24h \\ 
                     12h(1+3h)     & 8h^2+8h+c h& 10h \\
                     24h           & 10h        &6h+2c
      \end{pmatrix}.
\end{split}
\end{align}
The determinant $\det K_n$ is called the Kac determinant.
As conjectured in \cite{K:1979} and shown in \cite{FF:1982}, \cite{FF:1990}, 
it has the factored form
\begin{align}
\label{eq:Kac}
\det K_n(c,h)
=\prod_{\lambda \vdash n}2^{\ell(\lambda)}z_\lambda
 \times
 \prod_{\substack{r,s\in\bbZ_{\ge 1}\\ r s\le n}} 
 \big(h-h_{r,s}\big)^{p(n-r s)} .
\end{align}
Here $\ell(\lambda)$ is the length of the partition $\lambda$, 
$z_{\lambda}$ is given by
\begin{align}\label{eq:z_lambda}
z_{\lambda}\seteq\prod_{i\in\bbZ_{\ge1}}i^{m_i(\lambda)} m_i(\lambda) ! 
\quad \text{ with} \quad
m_i(\lambda)\seteq \#\{1\le i\le \ell(\lambda) \mid \lambda_j=i \},
\end{align}
and 
$p(n)\seteq\#\{\lambda\mid\lambda\vdash n\}$ 
is the partition number of $n$. 
To describe the zeros $h_{r,s}$, let us introduce a parametrization of $c$: 
\begin{align}
c=c(t)\seteq 13-6(t+t^{-1}). \label{eq:ct}
\end{align}
Then the rewritten form $h_{r,s}(t)\seteq h_{r,s}|_{c=c(t)}$ is given by 
\begin{align}
h_{r,s}(t)\seteq \dfrac{(r-s t)^2-(t-1)^2}{4t}. \label{eq:hrst}
\end{align}
In the following, we often use the pair $(c(t),h)$ for the highest weights. 
Note that by the symmetry 
\begin{align}\label{eq:symmetry:c}
c(t)=c(t^{-1}),
\end{align}
we have $M(c(t),h)=M(c(t^{-1}),h)$.

The definition \eqref{eq:sing:def} of the singular vector $v$ indicates 
that in the expansion 
$v=\sum_{\lambda\vdash n} c_\lambda L_{-\lambda}\ket{c,h}$ 
with respect to the basis \eqref{eq:L-lambda}, 
the set of coefficients $(c_\lambda)_{\lambda\vdash n}$ forms an eigenvector 
of the matrix \eqref{eq:Kac:matrix} with eigenvalue $0$.
Thus if a singular vector exists in $M(c,h)$, 
then the Kac determinant \eqref{eq:Kac} vanishes, 
i.e., the pair $(c,h)$ is expressed as $(c(t),h_{r,s}(t))$ 
for some $r,s\in\bbZ_{\ge 1}$ with $r s\le n$.

\subsection{Explicit formula of singular vectors}\label{subsec:efsv}

Several studies explored explicit forms of singular vectors. 
First, let us mention 

\begin{fct}[\cite{F:1994}]\label{fct:fuchs}
One can write the singular vector $\ket{\chi_{r,s}}$ 
on the Verma module $M(c(t),h_{r,s}(t))$ as 
\begin{align*}
\ket{\chi_{r,s}}=P_{r,s}(t)\ket{c(t), h_{r,s}(t)},
\end{align*}
with
\begin{align}\label{eq:Fuchs}
P_{r,s}(t)=L_{-1}^{rs}+\cdots\in U(\Vir_{-})\otimes\bbC[t,t^{-1}].
\end{align}
\end{fct}

The point is that the coefficients in \eqref{eq:Fuchs} are 
Laurent polynomials of $t$.

By direct calculations using the matrices \eqref{eq:Kac:matrix:eg},
one can obtain examples for small $r$ and $s$: 
\begin{align}\label{eq:sing:example}
\begin{split}
P_{1,2}(t)=&P_{2,1}(t^{-1})=L_{-1}^2-t          L_{-2},\\
P_{1,3}(t)=&P_{3,1}(t^{-1})
          =L_{-1}^3-4t         L_{-2}L_{-1}+2t     (2t-1)      L_{-3},\\
P_{1,4}(t)=&P_{4,1}(t^{-1})
          = L_{-1}^4-10t        L_{-2}L_{-1}^2+9t^2               L_{-2}^2\\
           &\phantom{P_{4,1}(t^{-1})=}
            +2t(12t-5)    L_{-3}L_{-1}-6t(6t^2-4t+1)L_{-4},\\
P_{2,2}(t)=&L_{-1}^4-2(t+t^{-1})L_{-2}L_{-1}^2+(t^2-2+t^{-2})     L_{-2}^2\\
           &-2(t-3+t^{-1})L_{-3}L_{-1}-3(t-2+t^{-1})L_{-4}.
\end{split}
\end{align}
In these examples the condition \eqref{eq:Fuchs} is clearly satisfied. 

As another simple remark, we have the equality 
\begin{align}\label{eq:t-t^-1}
P_{r,s}(t)=P_{s,r}(t^{-1}), 
\end{align}
which is an easy consequence of the symmetries \eqref{eq:symmetry:c} and 
\begin{align*}
h_{r,s}(t)=h_{s,r}(t^{-1}).
\end{align*}

In the late 1980s and the early 1990s, a series of works tried to write down 
$P_{r,s}(t)$ explicitly. 
\cite{BS:1988} gave an explicit formula for $P_{1,s}(t)$.
\cite{BFIZ:1991} gave an algorithm for constructing general 
$P_{r,s}(t)$ from $P_{1,s}(t)$. 
See also \cite[\S 8.A]{FMS:1997} for this algorithm.
\cite{K:1991} gave a formula of $P_{r,s}(t)$ 
using `analytic continuation' of $P_{1,s}(t)$.
See also \cite{F:1994} for a mathematically rigorous treatment 
of this `analytic continuation'. 
Although these formulas for $P_{r,s}(t)$ were enough for several studies 
on representation theory (see e.g. \cite{K:1992}), 
an `explicit' formula for $P_{r,s}(t)$ could not be obtained.

The paper \cite{MY:1995} shed a new light on this problem. 
It was found that the following two objects coincide up to normalization: 
the integral expression of the Jack symmetric function $J_{(s^r)}^{(t)}$ 
\cite[Chap.VI \S 10]{M:1995}
and the expression of $\ket{\chi_{r,s}}$ 
in terms of the screening operators for the Feigin-Fuchs bosonization 
of Virasoro algebra \cite{FF:1982}. 
We will recall this topic in \S \ref{sect:boson}.
See Fact \ref{fct:sing:jack} in \S \ref{subsec:jack} 
for the precise statement.

\subsection{Norm of logarithmic primary}\label{subsec:nlp}

Let us define an anti-homomorphism
\begin{align*}
\dagger: U(\Vir_{-}) \to U(\Vir_{+}), \quad L_{-n}\mapsto L_n.
\end{align*}
We will also denote this map as $L_{-n}^\dagger=L_n$.
Note that $(L_{-\lambda})^\dagger=L_\lambda$ under the notations 
\eqref{eq:L-lambda} and \eqref{eq:L+lambda}.
The anti-homomorphism $\dagger$ naturally induces a linear map 
$M(c,h)\to M^*(c,h)$, which is also written by $\dagger$. 
Note that $(\ket{c,h})^\dagger=\bra{c,h}$. 
We define $\bra{\chi_{r,s}}\seteq (\ket{\chi_{r,s}})^\dagger$.

For an element $v\in M(c,h)$, the norm of $v$ is defined to be 
\begin{align*}
v^\dagger\cdot v,
\end{align*}
where $\cdot$ is the contravariant form \eqref{eq:contrav}. 
For example, by the definition of singular vector \eqref{eq:sing:def}, 
it is obvious that 
\begin{align*}
\left<\chi_{r,s} | \chi_{r,s}\right>
=\bra{c(t),h_{r,s}(t)} P_{r,s}^\dagger(t) P_{r,s}(t) \ket{c(t),h_{r,s}(t)}=0.
\end{align*}

In \cite{IMM:1992} a curious observation was given on the norm 
\begin{align*}
 N_{1,s}(t,h)\seteq \bra{c(t), h} P_{1,s}^\dagger(t) P_{1,s}(t) \ket{c(t), h}
\end{align*}
of the vector $P_{1,s}(t) \ket{c(t), h}$.
They obtained a formula 
\begin{align}
\label{eq:Nr1}
\begin{split}
&N_{1,s}(t,h)=(h-h_{1,s}(t))\cdot R_{1,s}(t)+O((h-h_{1,s}(t))^2),\\
&R_{1,s}(t)\seteq 2 s! (s-1)! \prod_{k=1}^{s-1}(k^2 t^2-1).
\end{split}
\end{align}
A proof (containing some physical discussion) 
of this factor $R_{1,s}(t)$ was given in \cite{IMM:1992}.

There was a several year gap between the studies of this kind of norm of 
 $P_{r,s}(t) \ket{c(t), h}$. 
One of the reasons why such a calculation did not attract so much interest 
may be, as a simple matter, that there was no necessity.

In the early 2000s, a revival of the Liouville field theory occurred. 
Among several important papers, 
\cite[\S 6]{Z:2004} observed a generalization of the formula \eqref{eq:Nr1}. 
Let us denote
\begin{align}\label{eq:Nrs:def}
 N_{r,s}(t,h)\seteq \bra{c(t), h} P_{r,s}^\dagger(t) P_{r,s}(t) \ket{c(t), h}.
\end{align}
Using \eqref{eq:sing:example}, we can calculate some examples:
\begin{align*}
N_{1,1}(t,h)
=&2(h-h_{1,1}(t)),
\\
N_{1,2}(t,h)
=&4(t^2-1)(h-h_{1,2}(t))+8(h-h_{1,2}(t))^2,
\\
N_{1,3}(t,h)
=&24(t^2-1)(4t^2-1)(h-h_{1,3}(t))
\\
 &+8(16t^2-9)(h-h_{1,3}(t))^2+48(h-h_{1,3}(t))^3,
\\
N_{1,4}(t,h)
=&288(t^2-1)(4t^2-1)(9t^2-1)(h-h_{1,4}(t))
\\
 &+16(594t^4-481t^2+66)(h-h_{1,4}(t))^2
\\
 &+128(25t^2-9)(h-h_{1,4}(t))^3+384(h-h_{1,4}(t))^4
\\
N_{2,2}(t,h)
=&-8(t^2-1)(t^2-4)(t^{-2}-1)(t^{-2}-4)(h-h_{2,2}(t))
\\
 &+16(2t^{-4}-33t^{-2}+91-33t^2+2t^4)(h-h_{2,2}(t))^2
\\
 &+128(t^2-7+t^{-2})(h-h_{2,2}(t))^3+384(h-h_{2,2}(t))^4.
\end{align*}
Then it was conjectured that 
\begin{align}
\label{eq:Nrs}
\begin{split}
&N_{r,s}(t,h)\stackrel{?}{=}
 (h-h_{r,s}(t))\cdot R_{r,s}(t)+O((h-h_{r,s}(t))^2),\\
&R_{r,s}(t)\seteq 2 \prod_{\substack{(k,l) \in \bbZ^2, \\ 
                                    1-r \le k \le r,  \ 
                                    1-s \le l \le s,  \\
                                    (k,l)\neq(0,0),(r,s).}} 
                    (k t^{-1/2}+l t^{1/2}).
\end{split}
\end{align}
(Note that in \cite{Z:2004} 
 $N_{r,s}(t,h)$ is not expanded with respect to $h$, 
 but $\alpha$ given in \eqref{eq:FF:ch}, 
 the Heisenberg counterpart of the highest weight.)

The element $P_{r,s}(t)\ket{c(t),h}$ is named the logarithmic primary 
in \cite{Z:2004}, 
so that it is natural to call $N_{r,s}(t)$ 
the norm of logarithmic primary. 
This norm is the main object in this paper. 
The expression \eqref{eq:Nrs} is a generalization of \eqref{eq:Nr1}.
A physical derivation of the factor $R_{r,s}(t)$ was shown in \cite{Z:2004} 
based on the theory of Liouville field 
on the Poincar\'{e} disk \cite{ZZ:2001},
but it seems to lack mathematically rigorous arguments.  
An analogous explanation for the SUSY Liouville field theory was given in 
\cite{BZ:2006}, but no mathematical proof seems to exist.

As indicated in the last line of \cite[\S 6]{Z:2004}, 
the factor $R_{r,s}(t)$ resembles the dominator of the factor appearing in 
the recursive formula of conformal block given 
in \cite{Z:1984}, \cite{Z:1987} and \cite{ZZ:1996}.

\subsection{AGT relation}\label{subsec:AGT}

The resemblance of the factor $R_{r,s}(t)$, 
with the factor appearing in the recursive formula, 
was recently rediscovered in the context of AGT relations/conjectures.

The original AGT conjecture \cite{AGT:2010} states an equivalence between 
the Liouville conformal blocks and the Nekrasov partition functions 
for $N=2$ supersymmetric $\SU(2)$ gauge theories \cite{N:2003}.
In \cite{G:2009} degenerated versions of the conjecture were proposed.
As the most simplified case, 
it was conjectured that the inner product of a certain element 
in the Verma module of Virasoro algebra coincides 
with the Nekrasov partition function 
for the four-dimensional pure $\SU(2)$ gauge theory.

The element considered is a kind of Whittaker vector 
in the Verma module of the Virasoro algebra,
and now called Gaiotto state.
Let us recall its definition. 
Fix a non-zero complex number $\Lambda$.
Consider the completed Verma module $\widehat{M}(c,h)$ of $M(c,h)$, 
where the completion is done with respect to the $L_0$-weight 
gradation $M(c,h)=\oplus_{n\in\bbZ_{\ge0}}M(c,h)_n$.
An element $\ket{G}\in\widehat{M}(c,h)$ is called a Gaiotto state if 
\begin{align*}
L_1\ket{G}=\Lambda^2 \ket{G},\quad 
L_n\ket{G}=0\ (n>1).
\end{align*}
We normalize $\ket{G}$ by the condition 
\begin{align*}
\ket{G}=\ket{c,h}+\cdots,
\end{align*}
which means that 
the homogeneous component of $\ket{G}$ in $M(c,h)_0$ is $\ket{c,h}$.

On the other hand, the pure $\SU(r)$ gauge Nekrasov partition function 
has the next combinatorial expression 
(which can be considered as the definition of the partition function). 
Let $x,\ep_1,\ep_2,\vect{a}=(a_1,a_2,\ldots,a_r)$ be indeterminates.
\begin{align}
\label{eq:Nekrasov}
\begin{split}
Z^{\rk=r}(x;\ep_1,\ep_2,\vect{a})
\seteq
&\sum_{\vect{Y}}
\dfrac{ x^{|\vect{Y}|} }
{\prod_{1\le\alpha,\beta\le r} 
 n_{\alpha,\beta}^{\vect{Y}}(\ep_1,\ep_2,\vect{a})},
\\
n_{\alpha,\beta}^{\vect{Y}}(\ep_1,\ep_2,\vect{a})
\seteq
&\prod_{\square\in Y_\alpha}
 [-\ell_{Y_\beta}(\square)\ep_1+(a_{Y_\alpha}(\square)+1)\ep_2
 +a_\beta-a_\alpha]
\\
\times
&\prod_{\blacksquare\in Y_\beta}
 [(\ell_{Y_\alpha}(\blacksquare)+1)\ep_1- a_{Y_\beta}(\blacksquare)\ep_2
 +a_\beta-a_\alpha].
\end{split}
\end{align}
Here $\vect{Y}=(Y_1,Y_2,\ldots,Y_r)$ is an $r$-tuple of partitions, 
$|\vect{Y}|\seteq |Y_1|+|Y_2|+\cdots+|Y_r|$,
and $a_{Y}(\square)$, $\ell_{Y}(\square)$  are 
the arm and the leg of the box $\square$ with respect to $Y$ 
as will be defined in \eqref{eq:armleg}.

Now, the statement of the simplest case of Gaiotto conjectures is 
\begin{align}\label{eq:G:conj}
\langle G | G \rangle \stackrel{?}{=}Z^{\rk=2}(x;\ep_1,\ep_2,\vect{a}).
\end{align}
The parameters in both hand sides 
are related as in Table \ref{table:VirNek}.
\begin{table}[htbp]
\centering
\begin{tabular}[c]{cc}
Virasoro & Nekrasov
\\
\hline
$c$ & $13+6(\ep_1/\ep_2+\ep_2/\ep_1)$
\\
$h$ & $\big((\ep_1+\ep_2)^2-(a_2-a_1)^2\big)/4 \ep_1\ep_2$
\\
$\Lambda$ & $x^{1/4}/(\ep_1\ep_2)^{1/2}$
\end{tabular}
\vskip 1em
\caption{Parameter correspondence}
\label{table:VirNek}
\end{table}

Let us also mention the work of \cite{MMM:2009}, which shows 
\begin{align*}
\langle G | G \rangle
=\sum_{n=0}^\infty \Lambda^{4n}(K_n^{-1})_{(1^n),(1^n)},
\end{align*}
where $K_n^{-1}$ is the inverse matrix of \eqref{eq:Kac:matrix}, 
and the index `$(1^n),(1^n)$' indicates the position of the element 
of this inverse matrix 
(recall that the matrix $K_n$ is indexed by partitions of $n$).
Thus the conjecture \eqref{eq:G:conj} is equivalent to  
\begin{align}
(K_n^{-1})_{(1^n),(1^n)}(c,h)\stackrel{?}{=}
(\ep_1\ep_2)^{4n}Z_n(\ep_1,\ep_2;\vect{a}).
\label{eq:G:conj:2}
\end{align}

\subsection{Proving AGT relation via recursive formula}\label{subsec:prfAGT}

There exist several strategies for proving AGT conjectures. 
As for the pure gauge version  \eqref{eq:G:conj:2}, 
one of the strategies is to show that both sides satisfy the same 
recursive formula with respect to $n$.
This strategy was first proposed by \cite{P:2009}.

Later, the paper \cite{FL:2010} executed the procedure 
in case of the  adjoint matter theory. 
They used an integral expression of the Nekrasov partition function and 
showed that it satisfies a recursive formula. 
The same recursive formula for the Virasoro side was 
then derived from the Zamolodchikov recursive formula for conformal block 
by limiting procedure. 
 
Similar arguments were given in \cite{HJS:2010},
where the cases of $N_f=0,1,2$ (the number of matter fields) were treated. 
Here, we only recall the case $N_f=0$, that is, the pure gauge theory case.
It was shown that 
$z_n\seteq (\ep_1\ep_2)^{4n}Z_n(\ep_1,\ep_2;\vect{a})$ 
satisfies the  recursive formula:
\begin{align}\label{eq:rec:z}
z_n(t,h)=\delta_{0,n}+
\sum_{\substack{(r,s) \in \bbZ^2_{>0}, \\ 1 \le r s \le n}}
\dfrac{R_{r,s}(t)^{-1} z_{n-rs}(t,h_{r,s}(t)+r s)}{h-h_{r,s}(t)}.
\end{align}
Here the formula is written in the Virasoro parameter $(t,h)$.
In order to see it in the Nekrasov parameter, 
one needs to rewrite parameters by Table \ref{table:VirNek} 
and $c=c(t)$ defined in \eqref{eq:ct}.

On the Virasoro side, it was stated that 
$f_n(t,h)\seteq (K_n^{-1})_{(1^n),(1^n)}(c(t),h)$ 
satisfies the following recursive formula:
\begin{align}\label{eq:rec:c}
f_n(t,h)=\delta_{0,n}+
\sum_{\substack{(r,s) \in \bbZ^2_{>0}, \\ 1 \le r s \le n}}
\Big[\lim_{h\to h_{r,s}(t)} \dfrac{N_{r,s}(t,h)}{h-h_{r,s}(t)} \Big]^{-1}
\dfrac{f_{n-r s}(t,h_{r,s}(t)+r s)}{h-h_{r,s}(t)}.
\end{align}
Actually, one can obtain this formula by considering the 
Jantzen filtration (see \cite{J:1979}, \cite{FF:1990} and \cite{ES:1985}) 
on $M(c,h)$ and by some calculation on $K_n$. 
Thus from the comparison of \eqref{eq:rec:z} and \eqref{eq:rec:c},
the verification of the conjecture \eqref{eq:G:conj:2} is reduced to 
the conjecture \eqref{eq:Nrs}.

However, it seems that the justification of \eqref{eq:Nrs}
has not been discussed so far.
Note that in the argument of AGT relations, 
the formula of $R_{r,s}(t)$ \eqref{eq:Nrs} is cited without proof, 
only with the remark to the paper \cite{Z:2004}.
See \cite[(1.13a)]{FL:2010} and \cite[p. 7]{HJS:2010}.

This kind of `normalization factor' always appears in the recursive formula 
of conformal blocks and its degenerated versions.
See e.g. 
\cite[(5.40)]{ZZ:1990} (Liouville conformal block), 
\cite[(35)]{BZ:2006} (SUSY Liouville case),
\cite{HJS:2006}, \cite{HJS:2008a} and \cite{HJS:2008b} 
(also SUSY Liouville case).

Therefore, we consider that 
it is valuable to give a proof of the formula \eqref{eq:Nrs} for $R_{r,s}$.

\subsection{Main result of this paper}\label{subsect:main}

The main result is a rigorous proof of \eqref{eq:Nrs}.
Let us rephrase the statement.

\begin{thm}\label{thm:main}
Let 
\begin{align*}
P_{r,s}(t)=L_{-1}^{r s}+\cdots \in U(\Vir_{-})\otimes\bbC[t,t^{-1}]
\end{align*}
be the element which generates the singular vector 
\[\ket{\chi_{r,s}}=P_{r,s}(t)\ket{c(t),h_{r,s}(t)}\] 
in $M(c(t),h_{r,s}(t))$.
Define
\begin{align*}
 N_{r,s}(t,h)\seteq \bra{c(t),h} P_{r,s}^\dagger(t) P_{r,s}(t) \ket{c(t),h}.
\end{align*}
Then $N_{r,s}(t,h)$ has the form 
\begin{align*}
&N_{r,s}(t,h)=(h-h_{r,s}(t))\cdot R_{r,s}(t)+O((h-h_{r,s}(t))^2),\\
&R_{r,s}(t)\seteq 2 \prod_{\substack{(k,l) \in \bbZ^2, \\ 
                                    1-r \le k \le r,  \ 
                                    1-s \le l \le s,  \\
                                    (k,l)\neq(0,0),(r,s).}} 
                    (k t^{-1/2}+l t^{1/2}).
\end{align*}
\end{thm}

Let us explain the content of this paper.
\S 2 is devoted to the preliminaries on bosonization and symmetric functions,
which are crucial tools in our discussion.
In \S 3 we give the proof of the main Theorem \ref{thm:main}.
Since our argument is rather complicated, 
the outline is explained in the beginning of this section.
We end this paper with \S \ref{sec:rmk} giving some remarks on possible 
generalizations and the related works.


\section{Preliminaries on bosonization}\label{sect:boson}

\subsection{Bosonization and singular vectors}

Let us recall the Feigin-Fuchs bosonization of Virasoro algebra quickly. 
Consider the Heisenberg algebra $\calH$ generated by $a_n$ ($n\in\bbZ$) 
with the relation
\begin{align*}
[a_m,a_n]=m\delta_{m+n,0}.
\end{align*}
For a fixed $\rho\in\bbC$, consider the correspondence
\begin{align}
L_n \mapsto \calL_n
 \seteq\dfrac{1}{2}\sum_{m\in\bbZ}\no{a_m a_{n-m}}-(n+1)\rho a_n,
\quad
C \mapsto 1-12\rho^2,
\label{eq:FF}
\end{align}
where the symbol $\no{\ }$ means the normal ordering.
This correspondence determines a well-defined morphism 
\begin{align*}
\varphi: U(\Vir) \to\widehat{U}(\calH).
\end{align*}
Here $\widehat{U}(\calH)$ is 
the completion of the universal enveloping algebra $U(\calH)$ 
in the following sense 
(see also \cite{FFr:1996}, \cite{FrB:2004} and \cite{Fr:2007}). 
For $n\in\bbZ_{\ge0}$, 
let $I_n$ be the left ideal of $U(\calH)$ generated by all polynomials in $a_m$ ($m\in\bbZ_{\ge1})$ 
of degrees greater than or equal to $n$ 
(where we defined the degree by $\deg a_m\seteq m$).
Then we define 
\begin{align*}
\widehat{U}(\calH)\seteq\varprojlim_n \widehat{U}(\calH)/I_n.
\end{align*}

Next we recall the functorial correspondence of the representations. 
First let us define the Fock representation $\calF_\alpha$ of $\calH$.
$\calH$ has the triangular decomposition of 
$\calH=\calH_{+}\oplus\calH_{0}\oplus\calH_{-}$
with $\calH_{\pm}\seteq \oplus_{\pm n\in\bbZ_{\ge 1}}\bbC a_n$ and 
$\calH_{0}\seteq \bbC a_0$.
Let $\bbC_\alpha=\bbC\ket{\alpha}_\calF$ 
be the one-dimensional representation of 
$\calH_{0}\oplus\calH_{+}$ with the action 
$a_0\ket{\alpha}_\calF=\alpha \ket{\alpha}_\calF$ and 
$a_n\ket{\alpha}_\calF=0$ ($n\in\bbZ_{\ge1}$).
Then the Fock space $\calF_\alpha$ is defined to be 
\begin{align*}
\calF_\alpha\seteq
\Ind_{\calH_{0}\oplus\calH_{-}}^{\calH}\bbC_\alpha
\end{align*}
It has a weight decomposition  
\begin{align}\label{eq:F:deg}
\calF_\alpha=\oplus_{n\in\bbZ_{\ge 0}}\calF_{\alpha,n},
\end{align} 
where each weight space $\calF_{\alpha,n}$ has a basis 
\begin{align}\label{eq:a:base}
\{a_{-\lambda}\ket{\alpha}_\calF \mid \lambda\vdash n \}
\end{align} 
with $a_{-\lambda}\seteq a_{-\lambda_k}\cdots a_{-\lambda_1}$ 
for a partition $\lambda=(\lambda_1,\ldots,\lambda_k)$.
Then the action of $\widehat{U}(\calH)$ on $\calF_\alpha$ is well-defined.

Similarly the dual Fock space $\calF^*_\alpha$ is defined to be 
$\Ind_{\calH_{0}\oplus\calH_{-}}^{\calH }\bbC_\alpha^*$,
where $\bbC_\alpha^*=\bbC\cdot{}_\calF\bra{\alpha}$ 
is the one-dimensional representation of 
$\calH_{0}\oplus\calH_{-}$ with the action 
${}_\calF\bra{\alpha}a_0=\alpha \cdot {}_\calF\bra{\alpha}$ and 
${}_\calF\bra{\alpha} a_{-n}=0$ ($n\in\bbZ_{\ge1}$).

Now we can state the bosonization of representation. 
\eqref{eq:FF} is compatible with the map 
\begin{align*}
\psi: M(c,h) \to \calF_\alpha,\quad 
L_{-\lambda}\ket{c,h}\mapsto \calL_{-\lambda}\ket{\alpha}_\calF
\end{align*}
with $\calL_{-\lambda}\seteq\calL_{-\lambda_1}\calL_{-\lambda_2}\cdots\calL_{-\lambda_k}$ 
for $\lambda=(\lambda_1,\lambda_2,\ldots,\lambda_k)$
and 
\begin{align}\label{eq:FF:ch}
c=1-12\rho^2,\quad h=\dfrac{1}{2}\alpha(\alpha-2\rho).
\end{align}
In other words, we have 
\begin{align*}
\psi(x v)=\varphi(x)\psi(v)\quad (x\in \Vir,\ v\in M(c,h))
\end{align*} 
under the parametrization \eqref{eq:FF:ch} 
of highest weights.

Note that from the parametrization \eqref{eq:ct}, \eqref{eq:hrst} 
and the correspondence \eqref{eq:FF:ch}, a singular vector occurs at 
\begin{align*}
c=1-12\rho(t)^2,\quad 
h=h_{r,s}=\dfrac{1}{2}\alpha_{r,s}(t)\big(\alpha_{r,s}-2\rho(t)\big)
\end{align*} 
with
\begin{align}\label{eq:hw:rs:boson}
\rho(t)\seteq\dfrac{1}{\sqrt{2}}(t^{-1/2}-t^{1/2}),\quad
\alpha_{r,s}(t)\seteq\dfrac{1}{\sqrt{2}}\big((r+1)t^{-1/2}-(s+1) t^{1/2}\big).
\end{align}

The Fock space $\calF_\alpha$ is naturally identified 
with the space of symmetric functions. 
In this paper, the term symmetric function  means the 
infinite-variable symmetric polynomial. 
To treat such an object rigorously, 
we follow the argument of \cite[Chap.I \S 2]{M:1995}.  
Let us denote by $\Lambda_n$ the ring of $n$-variable symmetric polynomials 
over $\bbZ$, 
and by $\Lambda_n^d$ the space of homogeneous symmetric polynomials 
of degree $d$. 
The ring of symmetric functions $\Lambda$ is defined as the inverse limit 
of the $\Lambda_n$ in the category of graded rings 
(with respect to the gradation by the degree $d$). 
We denote by $\Lambda_K\seteq \Lambda\otimes_\bbZ K$ 
the coefficient extension to the ring $K$.
Among several bases of $\Lambda$, the power sum symmetric function 
\begin{align*}
p_n=p_n(x)\seteq \sum_{i\in\bbZ_{\ge1}} x_i^n,\quad
p_\lambda\seteq p_{\lambda_1}p_{\lambda_2}\cdots p_{\lambda_k},
\end{align*}
plays an important role. 
It is known that $\{p_\lambda\mid \lambda \vdash d\}$ is a basis of 
$\Lambda_\bbQ^d$, the space of homogeneous symmetric functions 
of degree $d$. 

Now following \cite{AMOS:1995}, 
we define the isomorphism between $\calF_\alpha$ and 
$\Lambda_{\bbC(t^{1/2})}$: 
\begin{align*}
\iota: \calF_\alpha \otimes \bbC(t^{1/2}) \to \Lambda_{\bbC(t^{1/2})},\quad 
v\mapsto 
{}_\calF\bra{\alpha}
\exp\Big(\dfrac{1}{\sqrt{2t}}\sum_{n=1}^\infty \dfrac{1}{n}p_n a_n\Big)v.
\end{align*}
Under this morphism, 
an element $a_{-\lambda}\ket{\alpha}_\calF$ of the base \eqref{eq:a:base} 
is mapped to 
\begin{align*}
\iota(a_{-\lambda}\ket{\alpha}_\calF)
= p_\lambda(x)/({\sqrt{2t}})^{\ell(\lambda)}.
\end{align*}
Since $\{p_\lambda\}$ is a basis of $\Lambda_\bbQ$, 
$\iota$ is actually an isomorphism.

Using the examples \eqref{eq:sing:example} and 
the Feigin-Fuchs bosonization \eqref{eq:FF}, 
one can calculate some examples of the images of singular vectors:
\begin{align*}
\iota\big(\varphi\big(P_{r,s}(t)\big)\ket{\alpha_{r,s}(t)}\big)
=\iota\circ\psi\big(P_{r,s}(t)\ket{c(t),h_{r,s}(t)}\big)
=\iota\circ\psi(\ket{\chi_{r,s}}).
\end{align*}
The result is 
\begin{align}
\label{eq:iotapsi:eg}
\begin{split}
&\iota\circ\psi(\ket{\chi_{1,1}})=J_{(1)}^{(t)}  \cdot (t^{-1}-1)      ,\\
&\iota\circ\psi(\ket{\chi_{2,1}})=J_{(1^2)}^{(t)}\cdot (t^{-1}-1)(2t^{-1}-1),\\
&\iota\circ\psi(\ket{\chi_{3,1}})
=J_{(1^3)}^{(t)}\cdot  (t^{-1}-1)(2t^{-1}-1)(3t^{-1}-1),\\
&\iota\circ\psi(\ket{\chi_{4,1}})
=J_{(1^4)}^{(t)}\cdot  (t^{-1}-1)(2t^{-1}-1)(3t^{-1}-1)(4t^{-1}-1),
\\
&\iota\circ\psi(\ket{\chi_{2,2}})
=J_{(2^2)}^{(t)}\cdot  (t^{-1}-1)(t^{-1}-2)(2t^{-1}-1)(2t^{-1}-2).
\end{split}
\end{align}
Here $J_\lambda^{(t)}$ is the integral Jack symmetric function,
the definition of which will be recalled in the next subsection.
Thus if one expresses $\ket{\chi_{r,s}}$ in terms of 
the Heisenberg generators $a_n$'s and identifies $a_{-n}$ with 
the power sum symmetric function $p_n$, 
then the expression of $\ket{\chi_{r,s}}$ is proportional 
to $J_{(s^r)}^{(t)}$.

\subsection{Jack symmetric function}
\label{subsec:jack}

Now we recall the definition and some properties of Jack symmetric function 
(see \cite[Chap.VI \S 10]{M:1995} and \cite{St:1989}). 
Let $t$ be an indeterminate 
\footnote{Our parameter $t$ is usually denoted by $\alpha$ in the literature, e.g., in \cite{M:1995}. We avoid using $\alpha$ since it is already defined to be the highest weight of the Heisenberg Fock space $\calF_\alpha$.}
and define an inner product on 
$\Lambda_{\bbQ(t)}$ by 
\begin{align*}
\langle p_\lambda,p_\mu \rangle_{t}\seteq
\delta_{\lambda,\mu}z_\lambda t^{\ell(\lambda)}.
\end{align*}
Here $z_\lambda$ is given in \eqref{eq:z_lambda}.
Then the monic Jack symmetric function $P_{\lambda}^{(t)}$ is determined 
uniquely by the following two conditions:
\begin{description}
\item[(i)] 
It has an expansion via monomial symmetric function $m_\nu$ in the form
\begin{align*}
P_{\lambda}^{(t)}
=m_\lambda+\sum_{\mu<\lambda}c_{\lambda,\mu}(t)m_\mu.
\end{align*}
Here $c_{\lambda,\mu}(t)\in\bbQ(t)$ and 
the ordering $<$ among the partitions is the dominance semi-ordering.

\item[(ii)]
The family of Jack symmetric functions is an orthogonal basis 
with respect to $\langle \cdot,\cdot \rangle_{t}$:
\begin{align*}
\langle P_\lambda^{(t)}, P_\mu^{(t)} \rangle_{t}=0 \quad 
\text{ if } \lambda\neq\mu.
\end{align*}
\end{description}

In order to define the integral Jack symmetric function $J_\lambda^{(t)}$, 
it is necessary to express the norm of $P_\lambda^{(t)}$. 
One can simply write it down using Young diagrams.
Following \cite{M:1995} we prepare several notations for diagrams here. 
To a partition $\lambda=(\lambda_1,\lambda_2,\ldots,\lambda_k)$, 
we associate the Young diagram, which is the set of boxes 
located at 
$\{(i,j)\in\bbZ^2\mid 1\le i \le k , 1\le j\le \lambda_i \}$. 
The coordinate $(i,j)$ is taken so that 
the index $i$ increases if one reads from top to bottom, 
and the index $j$ increases if one reads from left to right
(see e.g., Figure \ref{fig:442111}).
We will often identify a partition and its associated Young diagram.
\begin{figure}[htbp]
\unitlength 0.1in
\begin{center}
\begin{picture}(  8.0000,  12.0000)(  4.0000,-15.0000)
\special{pn 8}%
\special{pa 400 400}%
\special{pa 600 400}%
\special{pa 600 600}%
\special{pa 400 600}%
\special{pa 400 400}%
\special{fp}%
\special{pn 8}%
\special{pa 600 400}%
\special{pa 800 400}%
\special{pa 800 600}%
\special{pa 600 600}%
\special{pa 600 400}%
\special{fp}%
\special{pn 8}%
\special{pa 800 400}%
\special{pa 1000 400}%
\special{pa 1000 600}%
\special{pa 800 600}%
\special{pa 800 400}%
\special{fp}%
\special{pn 8}%
\special{pa 1000 400}%
\special{pa 1200 400}%
\special{pa 1200 600}%
\special{pa 1000 600}%
\special{pa 1000 400}%
\special{fp}%
\special{pn 8}%
\special{pa 400 600}%
\special{pa 600 600}%
\special{pa 600 800}%
\special{pa 400 800}%
\special{pa 400 600}%
\special{fp}%
\special{pn 8}%
\special{pa 600 600}%
\special{pa 800 600}%
\special{pa 800 800}%
\special{pa 600 800}%
\special{pa 600 600}%
\special{fp}%
\special{pn 8}%
\special{pa 800 600}%
\special{pa 1000 600}%
\special{pa 1000 800}%
\special{pa 800 800}%
\special{pa 800 600}%
\special{fp}%
\special{pn 8}%
\special{pa 1000 600}%
\special{pa 1200 600}%
\special{pa 1200 800}%
\special{pa 1000 800}%
\special{pa 1000 600}%
\special{fp}%
\special{pn 8}%
\special{pa 400 800}%
\special{pa 600 800}%
\special{pa 600 1000}%
\special{pa 400 1000}%
\special{pa 400 800}%
\special{fp}%
\special{pn 8}%
\special{pa 600 800}%
\special{pa 800 800}%
\special{pa 800 1000}%
\special{pa 600 1000}%
\special{pa 600 800}%
\special{fp}%
\special{pn 8}%
\special{pa 400 1000}%
\special{pa 600 1000}%
\special{pa 600 1200}%
\special{pa 400 1200}%
\special{pa 400 1000}%
\special{fp}%
\special{pn 8}%
\special{pa 400 1200}%
\special{pa 600 1200}%
\special{pa 600 1400}%
\special{pa 400 1400}%
\special{pa 400 1200}%
\special{fp}%
\special{pn 8}%
\special{pa 400 1400}%
\special{pa 600 1400}%
\special{pa 600 1600}%
\special{pa 400 1600}%
\special{pa 400 1400}%
\special{fp}%
\put(3.0, -3.0){\vector(1,0){4}}
\put(3.0, -3.0){\vector(0,-1){4}}
\put(7.6, -3.4){\makebox(0,0)[lb]{$j$}}%
\put(2.8, -8.5){\makebox(0,0)[lb]{$i$}}%
\end{picture}%
\end{center}
\caption{The Young diagram for $(4,4,2,1,1,1)$}
\label{fig:442111}
\end{figure}
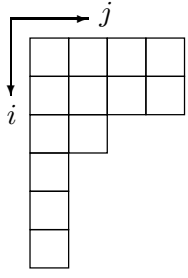

Now we define the arm and leg for a box $\square$ located at 
$(i,j)\in\bbZ^2_{\ge 1}$ with respect to $\lambda$ by 
\begin{align}
a_\lambda(\square)    \seteq \lambda_i-j,\quad 
\ell_\lambda(\square) \seteq \lambda^\vee_j-i. \label{eq:armleg}
\end{align}
Here $\lambda^\vee$ is the conjugate partition of $\lambda$, 
which is obtained by transposing the Young diagram of $\lambda$. 
(E.g., for $\lambda=(4,4,2,1,1,1)$ as in Figure \ref{fig:442111}, 
 we have $\lambda^\vee=(6,3,2,2)$.)
We also used the convention $\lambda_i=0$ for $i>\ell(\lambda)$ and 
$\lambda^\vee_j=0$ for $j>\lambda_1$. 
Thus $a_\lambda(\square)$ and $\ell_\lambda(\square)$ could be minus 
in general, 
although such a case does not occur in the norm of Jack symmetric functions.
(This generalized arm/leg is necessary for the definition 
of the Nekrasov partition function \eqref{eq:Nekrasov}.)

Using these combinatorial notations, 
one can write down the norm of monic Jack symmetric function as
\begin{align}\label{eq:jack:P:norm}
\langle P_\lambda^{(t)}, P_\lambda^{(t)} \rangle_{t} 
=\prod_{\square\in\lambda}
 \dfrac{t a_\lambda(\square)+\ell_\lambda(\square)+t}
       {t a_\lambda(\square)+\ell_\lambda(\square)+1},
\end{align}
where $\square\in\lambda$ means that the box $\square$ runs over 
the boxes in the Young diagram associated to $\lambda$.

Then the integral Jack symmetric function $J_\lambda^{(t)}$ 
is defined to be 
\begin{align}\label{eq:jack:JP}
J_\lambda^{(t)}\seteq
P_\lambda^{(t)} \cdot \prod_{\square\in\lambda}(t a_\lambda(\square)+\ell_\lambda(\square)+1)
\end{align}
It is known that for a partition $\lambda$ of $n$,
\begin{align*}
J_\lambda^{(t)}=\sum_{\mu\vdash n}u_{\lambda,\mu}(t) p_\lambda,\quad
u_{\lambda,(1^n)}(t)=1,\quad
u_{\lambda,\mu}(t)\in \bbZ[t].
\end{align*}
This is the origin of the word `integral' Jack symmetric function.

Finally, we can state the fact obtained in \cite{MY:1995}.

\begin{fct}\label{fct:sing:jack}
(1) \cite{MY:1995} The bosonization $\psi$ and the isomorphism $\iota$ 
    map the singular vector $\ket{\chi_{r,s}}$ 
    to the integral Jack symmetric function:
\begin{align*}
\iota\circ\psi(\ket{\chi_{r,s}}) \propto J_{(s^r)}^{(t)}.
\end{align*}
(2) \cite{SSAFR:2005} 
The proportional factor in the above equation is equal to 
\begin{align}\label{eq:B:rst}
B_{r,s}(t)\seteq \prod_{k=1}^r\prod_{l=1}^s (k t^{-1}-l).
\end{align}
\end{fct}

\section{The proof of Theorem \ref{thm:main}}

In this section, we will show our main theorem 
following the strategy of \cite{IMM:1992}.
Let us set 
\begin{align}\label{eq:Ars}
A_{r,s}(t)\seteq \lim_{h\to h_{r,s}(t)} R_{r,s}(t)/(h-h_{r,s}(t)).
\end{align} 

Our proof consists of the following steps.
\begin{description}
\item[Step 1]
Estimate the degree of $A_{r,s}(t)$ as the Laurent polynomial of $t$.
This step is executed with the help of asymptotic behavior of $P_{r,s}(t)$ 
studied in \cite{AF:1997}.
We show these arguments in \S \ref{subsec:step1}.

\item[Step 2]
Determine the set $S$ of the zeros of $A_{r,s}(t)$. 
This step is done in \S \ref{subsec:step2}, 
and it is divided into three sub-steps:
\begin{itemize}
\item Upper bound of $\# S$.
 The degree estimate in Step 1 gives the upper limit of 
 the number of zeros.

\item Lower bound of $\# S$.
 Next we show that $S$ includes a certain subset $S'$ using bosonization. 
 $S'$ is determined from a coefficient of Jack symmetric functions 
 appearing in the bosonization of singular vector. 
 Then we also show that $S$ is invariant under the action $t\mapsto -t$. 
 Thus $S$ contains $S'\cup -S'$, where $-S'\seteq\{-t\mid t\in S'\}$.

\item Third sub-step. 
 Since $\#(S'\cup -S')$ is equal to the upper limit, 
 $S$ should be equal to $S'\cup -S'$.
\end{itemize}

As the result of Step 1 and Step 2, 
$A_{r,s}(t)$ is determined up to a numerical factor.

\item[Step 3] Determine the numerical factor of $A_{r,s}(t)$.
We can determine this factor from the asymptotic behavior of $P_{r,s}(t)$ 
with respect to $t$. This step is done by direct calculation so that 
we omit the detail.
\end{description}

\subsection{Step 1. Degree estimate}\label{subsec:step1}

Let us recall the next fact:

\begin{fct}[\cite{AF:1997}]\label{fct:AF}
If one expands $P_{r,s}(t)$ as a Laurent polynomial of $t$, then
\begin{align*}
P_{r,s}(t)
=[(r-1)!]^{2s}L_{-r}^s t^{-(r-1)s}+\cdots+[(s-1)!]^{2r} L_{-s}^r t^{r(s-1)},
\end{align*}
where `$\cdots$' denotes the intermediate degrees in $t$.
\end{fct}

Now, for a Laurent series $a(t)=\sum_{k} a_k t^k$ of $t$, 
we define the maximum and minimum degrees of $a(t)$ by
\begin{align}\label{eq:def:maxmin}
\max\deg a(t)\seteq\max\{k\mid a_k\neq 0\},\quad
\min\deg a(t)\seteq\min\{k\mid a_k\neq 0\}.
\end{align}

\begin{lem}\label{lem:maxmin}
The maximum and minimum degrees of $A_{r,s}(t)$ are estimated as 
\begin{align*}
\max\deg A_{r,s}(t)\le 2r(s-1),\quad
\min\deg A_{r,s}(t)\ge-2(r-1)s.
\end{align*}
\end{lem}

\begin{proof}
First we treat $\max\deg A_{r,s}(t)$. 
Expand $\bra{c(t),h}L_{s}^{r} L_{-s}^{r} \ket{c(t),h}$ with respect 
to $h$ around $h_{r,s}(t)$ as 
\begin{align*}
\bra{c(t),h}L_{s}^{r} L_{-s}^{r}\ket{c(t),h}
=\sum_{k=0}^{r s-1} e_{k}^{(r,s)}(t) (h-h_{r,s}(t))^k.
\end{align*} 
By Fact \ref{fct:AF}, we have
\begin{align*}
\max\deg A_{r,s}(t)\le \max\deg \big( e_{1}^{(r,s)}(t) t^{2r(s-1)} \big).
\end{align*}
Since $\max\deg e_{1}^{(r,s)}(t)=0$ by Lemma \ref{lem:shap},
we have the result.

The estimate for  $\min\deg A_{r,s}(t)$ is obtained from that of 
$\max\deg A_{r,s}(t)$ 
by using the symmetry $t\mapsto t^{-1}$, $(r,s)\mapsto (s,r)$ \eqref{eq:t-t^-1}.
\end{proof}

\begin{cor}\label{cor:zero:1}
The number of zeros of $A_{r,s}(t)$ is at most $4r s-2r-2s$.
\end{cor}

\subsection{Step 2. Zero counting}\label{subsec:step2}

Recall the bosonization morphism $\varphi$ \eqref{eq:FF}, 
the correspondence \eqref{eq:FF:ch} 
and the highest weight \eqref{eq:hw:rs:boson} of Heisenberg algebra 
at which a singular vector exists.
We will use the highest weight shift of Heisenberg algebra:
\begin{align*}
\ve_{r,s}(t,\alpha)        \seteq \alpha-\alpha_{r,s}(t),\quad
\ve_{r,s}^\dagger(t,\alpha)\seteq \alpha-\alpha_{-r,-s}(t).
\end{align*}
Note that we have
\begin{align}\label{eq:dee}
h-h_{r,s}(t)=\dfrac{1}{2}\ve_{r,s}(t,\alpha)\ve_{r,s}^\dagger(t,\alpha).
\end{align}
under the correspondence \eqref{eq:FF:ch}.

For a partition $\lambda$ of $n$, 
define the degree of $a_{-\lambda}\in \bbC[a_{-1},a_{-2},\ldots,a_{-m}]$ 
by $\deg a_{-\lambda} \seteq |\lambda| =n$,
and denote by $\bbC[a_{-1},a_{-2},\ldots,a_{-m}]_n$ 
the subspace of homogeneous elements of degree $n$.
Then $\calF_{\alpha,n}$ defined in \eqref{eq:F:deg} is isomorphic 
to $\bbC[a_{-1},a_{-2},\ldots,a_{-n}]_n \ket{\alpha}_\calF$.
Similarly, we define the degree of $a_\lambda$ 
by $\deg a_\lambda\seteq |\lambda|$ and 
denote the homogeneous subspace by $\bbC[a_{1},a_{2},\ldots,a_{m}]_n$.

\begin{lem}\label{lem:Pg}
(1)
We have
\begin{align*}
\varphi(P_{r,s}(c(t),h)) \ket{\alpha}_\calF
=\ve_{r,s}^{\dagger}(t,\alpha)\Big[g_0(t)
 +\sum_{k=1}^{r s-1}(\ve_{r,s}(t,\alpha))^k g_k(t)\Big]
 \ket{\alpha}_\calF,
\end{align*}
with 
\begin{align*}
&g_0(t)\in\bbC[t^{\pm1/2}]\otimes\bbC[a_{-1},a_{-2},\ldots,a_{-r s}]_{r s},
\\ 
&g_k(t)\in\bbC[t^{\pm1/2}]\otimes\bbC[a_{-1},a_{-2},\ldots,a_{-r s+1}]_{r s} 
\quad (k\neq 0).
\end{align*}

(2)
We also have 
\begin{align*}
{}_\calF\bra{\alpha} \varphi(P_{r,s}^\dagger(c(t),h)) 
={}_\calF\bra{\alpha}
 \ve_{r,s}(\alpha,t)\Big[g_0^\dagger(t)
 +\sum_{k=1}^{r s-1} (\ve_{r,s}^\dagger (t,\alpha))^k g_k^\dagger(t)\Big],
\end{align*}
with 
\begin{align*}
&g_0^\dagger(t)
\in\bbC[t^{\pm1/2}]\otimes\bbC[a_1,a_{2},\ldots,a_{r s}]_{r s},
\\
&g_k^\dagger(t)
\in\bbC[t^{\pm1/2}]\otimes\bbC[a_1,a_{2},\ldots,a_{r s-1}]_{r s}
\quad (k\neq0).
\end{align*}
\end{lem}

The point is that $a_{-r s}$ (resp. $a_{r s}$) only appears 
in $g_0(t)$ (resp. $g_0^\dagger(t)$).
Before starting the proof, we show some examples.

\begin{eg}
\begin{align*}
&\varphi(P_{1,1}(c(t),h)) \ket{\alpha}_\calF
=\ve_{1,1}^\dagger(t,\alpha)\ket{\alpha}_\calF,
\\
&\varphi(P_{2,1}(c(t),h)) \ket{\alpha}_\calF
\\
&\phantom{=}
=\ve_{2,1}^\dagger(t,\alpha)
  \Bigl[
   (1-t)t^{-1}(\sqrt{2}t^{1/2} a_{-1}^2-a_{-2})+
   \ve_{2,1}(t,\alpha) a_{-1}^2\Bigr]
  \ket{\alpha}_\calF,
\\
&\varphi(P_{3,1}(c(t),h)) \ket{\alpha}_\calF
\\
&\phantom{=}
=\ve_{3,1}^\dagger(t,\alpha)\Big[
 (1-t)(2-t)t^{-2}(2t a_{-1}^3-3 \sqrt{2}t^{1/2}a_{-2}a_{-1}+2a_{-3})
\\
&\phantom{==\ve_{3,1}^\dagger(t,\alpha)\Big[}
 +\ve_{3,1}(t,\alpha)
  \bigl((3-2t)\sqrt{2}t^{-1/2} a_{-1}^3-(4t^{-1}-3)a_{-2}a_{-1}\bigr)
\\
&\phantom{==\ve_{3,1}^\dagger(t,\alpha)\Big[}
 +(\ve_{3,1}(t,\alpha))^2 a_{-1}^3\Big]\ket{\alpha}_\calF,
\\
&\varphi(P_{4,1}(c(t),h)) \ket{\alpha}_\calF
\\
&\phantom{=}
=\ve_{4,1}^\dagger(t,\alpha)
 \Big[(1-t)(2-t)(3-t)t^{-3}
 \bigl(2\sqrt{2}t^{3/2} a_{-1}^4
  +8\sqrt{2}t^{1/2}a_{-3}a_{-1}
\\
&\phantom{==\ve_{4,1}^\dagger(t,\alpha)\Big[(t-1)(2t-1)(3t-}
 +3\sqrt{2}t^{1/2}a_{-2}^2
 -12t a_{-2}a_{-1}^2-6a_{-4}\bigr)
\\
&\phantom{==\ve_{4,1}^\dagger(t,\alpha)\Big[}
+\ve_{4,1}(t,\alpha)
\bigl(2t^{-1} (11-12t+3t^2)a_{-1}^4
\\
&\phantom{==\ve_{4,1}^\dagger(t,\alpha)\Big[+\ve(t,\alpha)(2t}
-2\sqrt{2}t^{-3/2}(5-2t)(4-3t)a_{-2}a_{-1}^2
\\
&\phantom{==\ve_{4,1}^\dagger(t,\alpha)\Big[+\ve(t,\alpha)(2t}
+t^{-2}(9 -10t +3t^{2})a_{-2}^2
\\
&\phantom{==\ve_{4,1}^\dagger(t,\alpha)\Big[+\ve(t,\alpha)(2t}
+2t^{-2}(12-15t+4t^2)a_{-3}a_{-1}\bigr)
\\
&\phantom{==\ve_{4,1}^\dagger(t,\alpha)\Big[}
 +(\ve_{4,1}(t,\alpha))^2
 \bigl(3\sqrt{2}t^{-1/2}(2-t)a_{-1}^4-2t^{-1}(5-3t)a_{-2}a_{-1}^2\bigr)
\\
&\phantom{==\ve_{4,1}^\dagger(t,\alpha)\Big[}
 +(\ve_{4,1}(t,\alpha))^3 a_{-1}^4 \Big] \ket{\alpha}_\calF.
\end{align*}
\end{eg}

\begin{proof}[Proof of Lemma \ref{lem:Pg}]
(1)
$\varphi(P_{r,s}(c(t),h))\ket{\alpha}_\calF$ is a polynomial of 
degree $r s$ in terms of $\alpha$. 
It has at least one zero at $\alpha=\alpha_{-r,-s}$ by 
the discussion of \cite{FF:1982} (see also \cite{TK:1986, KM:1988}).
Thus the polynomial considered has the form 
\begin{align*}
(\alpha-\alpha_{-r,-s}(t))
\sum_{k=0}^{r s-1}(\alpha-\alpha_{r,s}(t))^k g_k(t) \ket{\alpha}_\calF,
\end{align*}
where 
$g_k(t)\in\bbC[t^{\pm1/2}]\otimes\bbC[a_{-1},a_{-2},\ldots,a_{-r s}]_{r s}$ 
for any $k$.
Note that the coefficients are in $\bbC[t^{\pm1/2}]$
by Fact \ref{fct:fuchs} and the definition of $\varphi$.

Now Lemma \ref{lem:a_-n} proved later means that the coefficient 
of $a_{-r s}$ has degree one as a polynomial of $\alpha$. 
Therefore  $a_{-r s}$ cannot appear in $g_k(t)$ for $k>1$. 
Thus we have the consequence.

(2) is similarly shown so that we omit the detail.
\end{proof}

\begin{lem}\label{lem:IMM}
The set of zeros in $A_{r,s}(t)=0$ includes 
the set of zeros in the coefficient of $a_{-r s}$ in $g_0(t)$.
\end{lem}

\begin{proof}
Precisely speaking, if the condition holds, then 
one can show that the state $\varphi(P_{r,s}(c(t),h))\ket{\alpha}_\calF$ 
vanishes identically.
In order to show it, 
it is enough to show by induction that 
$a_n \cdot \varphi(P_{r,s}(c(t),h))\ket{\alpha}_\calF=0$ 
for any $n$.
This is proved in \cite[Appendix B]{IMM:1992}.
\end{proof}


\begin{lem}\label{lem:a-rs}
The coefficient $c_\lambda(t)$ of $a_{-r s}$ in $g_0(t)$ is 
\begin{align*}
c_\lambda(t)
=\prod_{\substack{(k,l)\in \bbZ^2 \setminus\{(r,s)\} \\ 1\le k\le r,\ 1\le l\le s}}(k t^{-1}-l)
\cdot\prod_{\substack{(k,l)\in \bbZ^2 \setminus \{(0,0)\} \\ 0\le k \le r-1,\ 0\le l\le s-1}} (l t -k).
\end{align*}
\end{lem}

\begin{proof}
Note that by Fact \ref{fct:sing:jack} and Lemma \ref{lem:Pg} we have
\begin{align*}
 (\alpha_{r,s}(t)-\alpha_{-r,-s}(t))\cdot \iota(g_0(t)\ket{\alpha}_\calF)
=B_{r,s}(t) J_{(s^r)}^{(t)}.
\end{align*}
On the other hand, by Corollary \ref{cor:J:p} in \S \ref{subsec:appendix} 
we have
\begin{align*}
J_{(s^r)}^{(t)}=\sum_{\mu \vdash r s} \theta_{(s^r)}^{\mu}(t) p_\mu,
\quad
\theta_{(s^r)}^{(r s)}(t)
=\prod_{\substack{(k,l)\in \bbZ^2 \setminus \{(0,0)\} \\ 0\le k \le r-1,\ 0\le l\le s-1}} (l t-k).
\end{align*}
Since $\iota(a_{-r s}\ket{\alpha}_\calF)=p_{rs}/\sqrt{2t}$, we have
\begin{align*}
 (\alpha_{r,s}(t)-\alpha_{-r,-s}(t)) c_\lambda(t)/\sqrt{2t}  
=B_{r,s}(t)\cdot\theta_\lambda^{(n)}(t).
\end{align*}
Then an easy calculation shows the statement.
\end{proof}

\begin{cor}\label{cor:zero:pos}
$A_{r,s}(t,h)$ has zeros at 
\begin{align*}
S'\seteq
&\{t=k/l \mid 1\le k\le r,\ 1\le l \le s, \ (k,l)\neq(r,s)\}
\\
&\cup
\{t=k/l \mid 1\le k\le r-1,\ 1\le l \le s-1\}\quad
\text{(multiplicities included)}.
\end{align*} 
In particular, $A_{r,s}(t,h)$ has at least $\# S'=2r s-r-s$ zeros.
\end{cor}

\begin{proof}
This is the consequence of Lemma \ref{lem:IMM} and Lemma \ref{lem:a-rs}.
\end{proof}

\begin{lem}\label{lem:t-t}
$A_{r,s}(t)=A_{r,s}(-t)$.
\end{lem}

\begin{proof}
In fact one can show a stronger statement:
if one expands 
\begin{align}\label{eq:t-t}
N_{r,s}(t,h)=\sum_{k=1}^{r s} (h-h_{r,s}(t))^k n_{r,s,k}(t),
\end{align}
then we have $n_{r,s,k}(t)=n_{r,s,k}(-t)$. 
$A_{r,s}(t)$ is nothing but $n_{r,s,1}(t)$.
This is the consequence of the following 
Fact \ref{fct:t-t} stated in \cite[Theorem 1.13]{FF:1990} 
(for a proof, see also \cite{K:1992}).
\end{proof}

\begin{fct}[{\cite[Theorem 1.13]{FF:1990}}]\label{fct:t-t}
The anti-automorphism on $U(\Vir_-)[t,t^{-1}]$ 
defined by 
\begin{align*}
t\mapsto -t,\quad L_{-i}\mapsto (-1)^{i-1}L_{-i}
\end{align*}
takes $P_{r,s}(t)$ into itself.
\end{fct}

\begin{cor}\label{cor:zero:2}
Let us denote $-S'\seteq\{-t \mid t\in S'\}$.
Then $A_{r,s}(t)$ has zeros at $S'\cup (-S')$ (multiplicities included).
Thus $A_{r,s}(t)$  has at least $2\cdot \# S=4r s-2r-2s$ zeros.
\end{cor}

\begin{proof}
This is the consequence of Corollary \ref{cor:zero:pos} 
and Lemma \ref{lem:t-t}.
\end{proof}

\begin{cor}
$A_{r,s}(t)$ is equal to 
\[A'_{r,s}(t)\seteq\prod_{\substack{(k,l) \in \bbZ^2, \\ 
                                    1-r \le k \le r,  \ 
                                    1-s \le l \le s,  \\
                                    (k,l)\neq(0,0), (r,s).}} 
                    (k t^{-1/2}+l t^{1/2})\]
up to a numerical factor in $\bbC$.
\end{cor}

\begin{proof}
By Corollary \ref{cor:zero:1} and Corollary \ref{cor:zero:2}, 
we have known that the set $S$ of zeros of $A_{r,s}(t)$ is equal to 
$S'\cup(-S')$. 
Then the inequalities in Lemma \ref{lem:maxmin} should be equalities: 
$\max\deg A_{r,s}(t)=2r(s-1)$, $\min\deg A_{r,s}(t)=-2(r-1)s$.

Next note that $A'_{r,s}(t)\in\bbC[t,t^{-1}]$, that the set of its zeros is
equal to $S'\cup (-S')$, 
and that $\max\deg A'_{r,s}(t)=2r(s-1)$, $\min\deg A'_{r,s}(t)=-2(r-1)s$.
Thus $A_{r,s}(t)$ and $A'_{r,s}(t)$ should be proportional.
\end{proof}

\subsection{Lemmas for the proof}\label{subsec:appendix}

\subsubsection{A lemma on Shapovalov form}

\begin{lem}\label{lem:shap}
(1)
\begin{align*}
\bra{c,h}L_{s}^r L_{-s}^r\ket{c,h}
=s^r r! \prod_{k=1}^{r} [2h+s(k-1)+\dfrac{s^2-1}{12}c]
\end{align*}

(2) In the expansion
\begin{align*}
\bra{c(t),h}L_{s}^{r} L_{-s}^{r} \ket{c(t),h}
=\sum_{k=0}^s (h-h_{r,s}(t))^{k} e_{k}^{(r,s)}(t) ,
\end{align*}
the maximum and the minimum degrees 
(see \eqref{eq:def:maxmin} for the definitions) 
of $e_1^{(r,s)}(t)$ are
\begin{align*}
\max\deg e_1^{(r,s)}(t)=0,\quad
\min\deg e_1^{(r,s)}(t)=\begin{cases}0 & r=s \\ 1-r & r\neq s \end{cases}.
\end{align*}
\end{lem}

\begin{proof}
By the defining relations of $\Vir$ we have
\begin{align*}
\bra{c,h}L_{s}^{r} L_{-s}^{r}\ket{c,h}
=\bra{c,h}L_{s}^{r-1} L_{-s}^{r-1} \ket{c,h}
 \cdot \sum_{k=1}^r \Big[\dfrac{s^3-s}{12}c+2s(s k+h)\Big].
\end{align*}
Then an easy calculation gives the result of (1).

(2) is an immediate consequence of (1).
\end{proof}

\subsubsection{A lemma on bosonization}

\begin{lem}\label{lem:a_-n}
Let $\lambda$ be a partition of $n$.
In the expansion
\begin{align*}
\varphi(L_{-\lambda})\ket{\alpha}_\calF
=\sum_{\mu \vdash n} c^{\mu}_\lambda(t,\alpha) a_{-\mu} \ket{\alpha}_\calF
 \quad
 c^{\mu}_\lambda(t,\alpha) \in \bbC[\alpha,t^{\pm 1/2}],
\end{align*}
the coefficient $c^{(n)}_{\lambda}(t,\alpha)$ of $a_{-(n)} \ket{\alpha}_\calF$ 
is of degree one in terms of $\alpha$.
\end{lem}

\begin{proof}
This is an easy consequence of the bosonization $\varphi$:
\begin{align*}
&\varphi(L_{-2n})= (2n-1) \rho(t) a_{-2n}
+\dfrac{a_{-n}^2}{2}+a_{-n-1}a_{-n+1}+a_{-n-2}a_{-n+2}+
\\
&\hskip 5em
\cdots+a_{-2n+1}a_{-1}+a_{-2n}a_{0}+a_{-2n-1}a_{1}+\cdots,
\\
&\varphi(L_{-2n-1})= 2n \rho(t)  a_{-2n}
+a_{-n-1}a_{-n}+a_{-n-2}a_{-n+1}+
\\
&\hskip 6em
\cdots+a_{-2n}a_{-1}+a_{-2n-1}a_{0}+a_{-2n-2}a_{1}+\cdots.
\end{align*}

In fact, for the case $\ell(\lambda)=1$, i.e., $\lambda=(n)$, 
the above expression of $\varphi(L_{-n})$ gives 
$\varphi(L_{-n})\ket{\alpha}_\calF
=(\alpha+(n-1)\rho(t))a_{n}\ket{\alpha}_\calF+\cdots$,
which is the desired consequence.
 
For the case $k\seteq\ell(\lambda)>1$, 
set $\nu\seteq(\lambda_2,\lambda_3,\ldots,\lambda_{k})$.
Since 
\begin{align*}
\varphi(L_{-\nu})\ket{\alpha}_\calF
=\sum_{\mu\,\vdash n-\lambda_1} 
 c^{\mu}_{\nu}(t,\alpha)a_{-\mu}\ket{\alpha}_\calF,
\end{align*}
we find that the term 
$a_{-n}\ket{\alpha}_\calF$ in 
$\varphi(L_{-\lambda})\ket{\alpha}_\calF
=\varphi(L_{-\lambda_1})\varphi(L_{-\nu})\ket{\alpha}_\calF$ appears as 
\begin{align*}
a_{-n}a_{n-\lambda_1} \cdot 
c^{(n-\lambda_1)}_{\nu}(t,\alpha) a_{-n+\lambda_1} \ket{\alpha}_\calF
=(n-\lambda_1) c^{(n-\lambda_1)}_{\nu}(t,\alpha) a_{-n} \ket{\alpha}_\calF.
\end{align*}
By the induction hypothesis, we know that 
$c^{(n-\lambda_1)}_{\nu}(t,\alpha)$ is 
of degree one as a polynomial of $\alpha$. 
Thus the desired consequence holds.
\end{proof}

\subsubsection{A coefficient of Jack symmetric function}

\begin{fct}\label{fct:HSS}
If one expands the power symmetric function $p_n$ by 
the family of monic Jack symmetric functions $\{P_\lambda^{(t)}\}$,
then 
\begin{align*}
p_n=n t \sum_{\lambda \vdash n} 
     \prod_{\square\in\lambda}
      \dfrac{1}{t a_\lambda(\square)+\ell_\lambda(\square)+t}
     \prod_{\substack{(i,j)\in\lambda \\ (i,j)\neq(1,1)}}
           (t (j-1)-(i-1))
     \cdot P_\lambda^{(t)}.
\end{align*}
(For the notations of Young diagrams, 
 see \S \ref{subsec:jack}.)
\end{fct}
\begin{proof}
See \cite{HSS:1992}.
\end{proof}

\begin{cor}\label{cor:J:p}
Let $\lambda$ be a partition of $n$.
If one expands the integral Jack symmetric function $J_\lambda^{(t)}$
by the family of power symmetric function $\{p_\mu\}$ as
\[J_{\lambda}^{(t)}=\sum_{\mu\vdash n} \theta_\lambda^\mu(t)p_\mu\] 
then the following holds:
\begin{align*}
\theta_\lambda^{(n)}(t)
=\prod_{\substack{(i,j)\in\lambda \\ (i,j)\neq(1,1)}} ((j-1)t-(i-1)).
\end{align*}
\end{cor}

\begin{proof}
The coefficient $c_\lambda$ can be calculated by Fact \ref{fct:HSS}, 
the ratio \eqref{eq:jack:JP} of 
$P_\lambda^{(t)}$ and $J_\lambda^{(t)}$,
and the norm \eqref{eq:jack:P:norm} of $P_\lambda^{(t)}$.
\end{proof}


\section{Conclusion and Remarks}\label{sec:rmk}

The main result of this paper is Theorem \ref{thm:main}, 
a mathematical proof of \eqref{eq:Nrs}.
In our proof, important roles are played by Feigin-Fuchs  
bosonization and Jack symmetric functions. 
The technical but crucial point of our discussion is Fact \ref{fct:sing:jack} 
that the bosonized singular vector of Virasoro algebra 
is proportional to Jack symmetric function.
There are analogous phenomena of this fact in other algebras, such as 
$\calW$ algebras and the deformed Virasoro algebra.
So it may be possible to simulate Theorem \ref{thm:main} for these algebras.
Let us discuss the possibility of such extensions 
and also mention to the related AGT conjectures.

The paper \cite{AMOS:1995} showed that singular vectors of 
$\calW(\frksl_n)$ algebra are proportional to 
Jack symmetric functions associated to 
general (i.e., not necessarily rectangle) partitions.
But the proportional factor is not known so far, 
which is an obstruction to simulate our strategy to calculate 
the norm of logarithmic primaries in $\calW$ algebra case. 

Let us also mention to the $\SU(n)$ AGT conjecture, 
where $\calW(\frksl_n)$ algebra appears (see e.g. \cite{W:2008}).
In \cite{T:2009} a pure gauge AGT conjecture for $\calW(\frksl_3)$ algebra
was proposed. 
In this case, one can observe Zamolodchikov-type recursive formula, 
so that the strategy for the proof of the conjecture may be built.
However in order to execute this strategy, 
it is necessary to overcome the obstruction mentioned above . 
Note that the `finite analog' of this AGT conjecture is solved recently 
by methods in geometric representation theory \cite{BFRF:2010}.

Another possibility is the proof of the $q$-deformed/five-dimensional 
AGT conjecture proposed in \cite{AY:2009}. 
It is conjectured that 
the $K$-theoretic Nekrasov partition function coincides with the norm of 
the Gaiotto state in the Verma module of the deformed Virasoro algebra. 
In \cite{Y:2010} the author gave a recursive formula for the $K$-theoretic 
Nekrasov partition function. 
However, our knowledge on the singular vectors is not enough to 
give some proof of the recursive formula in the deformed Virasoro side.
It is known that the singular vectors of deformed Virasoro algebra are 
proportional to Macdonald symmetric functions associated to rectangle partitions \cite{SKAO:1996}, 
but their proportional factors are not known. 
Also the degree estimation of the norm of logarithmic primaries in the deformed case is not known.


\begin{ack}
The author is supported by JSPS Fellowships for Young Scientists (No.21-2241) 
and JSPS/RFBR joint project 
`Integrable system, random matrix, algebraic geometry and geometric invariant'.
Results in this paper were presented at the workshops 
`Topics on BC systems and AGT conjectures', Tokyo, September 2010, and 
`Synthesis of integrabilities in
the context of duality between strings and gauge theories', Moscow, 
September 2010. 
Thanks are due to the organizers and participants 
of these conferences for stimulating discussion.
The author also expresses gratitude to the adviser Professor K\={o}ta Yoshioka 
and Professor Yasuhiko Yamada for the valuable discussion.
\end{ack}


\end{document}